\documentclass[11pt]{article}

\usepackage{amsmath, amsfonts, amssymb}
\usepackage{mathrsfs}
\usepackage{theorem}

\newtheorem{thm}{Theorem}[section]
\newtheorem{prop}[thm]{Proposition}
\newtheorem{lem}[thm]{Lemma}
\newtheorem{cor}[thm]{Corollary}
\newtheorem{defn}[thm]{Definition}

{\theorembodyfont{\rmfamily} \newtheorem{rem}[thm]{Remark}}
{\theorembodyfont{\rmfamily} \newtheorem{exam}{Example}[section]}

\newcommand{\ra}{\rightarrow}
\newcommand{\dis}{\displaystyle}

\def\R{\mathbb R}

\def\N{\mathbb N}

\def\d{\text{\rm{d}}}
\def\E{\mathbb E}
\def\p{\mathbb P}

\def\la{\langle}
\def\raa{\rangle}
\def\La{\Lambda}
\def\veps{\varepsilon}

\def\diag{\mathrm{diag}}
\def\S{\mathcal M}
\def\rsp{\textbf{RSDP} }
\newcommand{\bxi}{\boldsymbol{\xi}}

\newcommand{\fin}{\hspace*{\fill}\rule{0.3em}{1ex}}
\newenvironment{proof}{{\bf \noindent Proof.}}{\fin}

\numberwithin{equation}{section}

\begin{document}

\title{Stability and recurrence of regime-switching diffusion processes\footnote{Supported in
 part by NSFC (No.11301030, No.11171024), 985-project and Beijing Higher Education Young Elite Teacher Project.}}

\author{Jinghai Shao\footnote{School of Mathematical Sciences, Beijing Normal University, Beijing, China. Email: shaojh@bnu.edu.cn}\ \  and \ Fubao Xi\footnote{School of Mathematics, Beijing Institute of Technology,
Beijing 100081, China. Email: xifb@bit.edu.cn.}}
\maketitle
\begin{abstract}
We provide some criteria on the stability of regime-switching diffusion processes. Both the state-independent and state-dependent regime-switching diffusion processes with switching in a finite state space and an infinite countable state space are studied in this work. We provide two methods to deal with switching processes in an infinite countable state space. One is a finite partition method based on the nonsingular M-matrix theory. Another is an application of principal eigenvalue of a bilinear form. Our methods can deal with both linear and nonlinear regime-switching diffusion processes. Moreover, the method of principal eigenvalue is also used to study the recurrence of regime-switching diffusion processes.
\end{abstract}
AMS subject Classification (2010):\  60J27, 60J60, 93E15, 60A10\\
\noindent \textbf{Keywords}: Regime-switching diffusions, Stability, State-dependent switching, Recurrence
\section{Introduction}

In this work, we shall study the stability of regime-switching
diffusion processes (for short, \textbf{RSDP}) which arise in financial engineering, wireless
communication and many other application fields. The stability of
such systems is of great interest and there has been a great deal of study in
this topic; see for example, \cite{BBG1996, BBG1999, FC, JC1990,
KZY2007, M1999, MYY2007, MY, M1990, YX2010, YZ} and references
therein. Especially, stability of linear or semi-linear type of such
systems has been investigated by \cite{BBG1999, JC1990, M1990} among
others.
These works generalized the Lyapunov's second method to deal with
the regime-switching diffusion processes. In particular, for linear
systems, some easily verifiable conditions were provided to ensure
the stability or instability of regime-switching processes. It is
well known that the construction of Lyapunov function is a rather
difficult task, so it is of great value to find some easily
verifiable conditions to ensure the existence of desired Lyapunov
functions. The main aims of this work are twofold: one is to find
some easily verifiable conditions to justify the stability or the
instability of \emph{nonlinear} regime-switching processes;
another is to generalize these conditions to study the stability of
\emph{state-dependent} regime-switching processes in an
\emph{infinite countable state space}. Up to our knowledge, there is few
result on the stability of regime-switching processes in an infinite
countable state space. Besides, there are some discussion on the stability of a
linearized system with the stability of the initial nonlinear system
in \cite[Chapter 7]{Kh}. There R. Khasminskii gave some positive
answer. On the other hand, we should note that there are essential
difference between the stability of nonlinear system with that of
linear system. On the stability of nonlinear control system, both
the well-known conjectures of Aizerman and Kalman were proven wrong
by counter-example.

The regime-switching process considered in this work consists two components, $(X_t,\La_t)$. The first component $(X_t)$ satisfies a stochastic differential equation with coefficients depending on the process $(\La_t)$; the second component $(\La_t)$ is a continuous time Markov chain on a finite or a countable space. One can view the process $(X_t)$ as a diffusion process in a random environment characterized by the process $(\La_t)$. The stability of $(X_t)$ in a random environment is more complicated than that of a diffusion process in a fixed environment. There are many examples (see \cite{MY} , \cite{YZ} concrete examples) to show that when $(X_t)$ is stable in some fixed environments and unstable in other fixed environments, one can make $(X_t)$ to be stable or unstable by choosing suitable switching rate, i.e. the $Q$-matrix of $(\La_t)$. In this work, we shall provide some on-off type criteria to show how the switching rate ($Q$-matrix) and the stability of $(X_t)$ in each fixed environment work together to determine the stability of the process $(X_t,\La_t)$. In this work, the stability  only  means the asymptotic stability in probability of the system $(X_t,\La_t)$. Other kinds of stability (for instance, $p$-stability) are left for further work.

To see the usefulness and sharpness of our criteria, let us consider the following one-dimensional nonlinear system.
\[\d X_t=b_{\La_t}\big(X_t^2\wedge |X_t|\big)\d t+\sigma_{\La_t}\big(X_t^2\wedge |X_t|\big)\d B_t,\]
where $(\La_t)$ is a continuous time Markov chain in a finite state space $\S$, $b_i$, $\sigma_i$ are constants for each $i\in \S$. Here we use $x^2\wedge |x|:=\min\{x^2,|x|\}$ to guarantee the solution of previous SDE to be nonexplosive, which does not impact the nonlinearity of the system near $0$. Let $\mu$ be the invariant measure of $(\La_t)$. Applying our criteria, we show that $x=0$ is asymptotically stable in probability if $\sum_{i\in \S}\mu_ib_i<0$, and is unstable in probability if $\sum_{i\in \S}\mu_ib_i>0$ (see Corollary \ref{cor-1} below).

In \cite{Sh14a,Sh14b}, the recurrence of \rsp in Wasserstein distance and in total varitional norm has been studied.  In this work, we develop their ideas to study the stability of \textbf{RSDP}. We mention some difference between the study of ergodicity  and that of stability for \rsp compared with \cite{Sh14a,Sh14b}. Both of these studies try to construct a Lyapunov function $V(x,i)$ such that $\dis\mathscr A\,V(x,i)\leq 0$, where $\mathscr A$ denotes the infinitesimal generator of \textbf{RSDP}. When we study the ergodicity, the key point is the behavior of $V(x,i)$ in the neighborhood of  $\infty$, but to study the stability, the key point is the behavior of $V(x,i)$ in the neighborhood of $0$. Note the following obvious fact: $\lim_{x\ra \infty}\frac{x^2}{x}=\infty$ but $\lim_{x\ra 0}\frac{x^2}{x}=0$. Consequently,  the dominant terms in $\mathscr A\,V(x,i)$ are different in these different situations.

This work is organized as follows. In Section 2, we first provide two kinds of criteria for stability of \rsp in a finite state space, which are based separatively on the Friedholm alternative and nonsingular M-matrix theory. Then in subsection 2.2, we extend the criterion in terms of M-matrix to deal with \rsp in a countable state space by putting forward a finite partition method. This method can also be used to deal with state-dependent \textbf{RSDP}. In Section 3, we provide some criteria for stability of \rsp in terms of the principal eigenvalue of a bilinear form. The method can deal with \rsp in a finite or countable state space. Compared with the method in terms of M-matrix theory, this criterion can deal with switching process in a countable state space without assuming the boundedness of the jumping rates. In Section 4, we apply the principal eigenvalue to study the recurrence of \textbf{RSDP}. Moreover, we provide a lower estimate of the principal eigenvalue defined in our work, which generalizes the corresponding result for the principal eigenvalue of Dirichlet forms in \cite{Ch00}. Two concrete examples are constructed to show the usefulness of this method in dealing with \rsp with unbounded jumping rates.

\section{Criteria based on M-matrix theory and Friedholm alternative}

The regime-switching diffusion process studied in this work can be viewed as a number of diffusion processes modulated by a random switching device or as a diffusion process which lives in a random environment. More precisely, \rsp  is a two-component process $(X_t,\La_t)$, where $(X_t)$ describes the continuous dynamics, and $(\La_t)$ describes the random switching device. $(X_t)$ satisfies the stochastic differential equation (for short, SDE)
\begin{equation}\label{1.1}
\d X_t=b(X_t,\La_t)\d t+\sigma(X_t,\La_t)\d B_t,\ X_0=x\in \R^d,
\end{equation}
where $(B_t)$ is a Brownian motion in $\R^d$, $d\geq 1$, $\sigma$ is $d\times d$-matrix, and $b$ is a vector in $\R^d$.
While for each fixed $x\in \R^d$, $(\La_t)$ is a continuous time Markov chain on the state space $\S=\{1,2,\ldots,N\}$, $2\leq N\leq \infty$, satisfying
\begin{equation}\label{1.2}
\p(\La_{t+\delta}=l|\La_t=k, X_t=x)=\left\{\begin{array}{ll} q_{kl}(x)\delta+o(\delta), &\text{if}\ k\neq l,\\
                                       1+q_{kk}(x)\delta+o(\delta), & \text{if}\ k=l,
                         \end{array}\right.
\end{equation}
for $\delta>0$.
The $Q$-matrix $Q_x=(q_{kl}(x))$ is irreducible and conservative for each $x\in \R^d$.
If the $Q$-matrix $(q_{kl}(x))$ does not depend on $x$, then $(X_t,\La_t)$ is called a state-independent \textbf{RSDP}; otherwise, it is called a state-dependent one. When $N$ is finite, namely, $(\La_t)$ is a Markov chain on a finite state space, we call $(X_t,\La_t)$ a \rsp  in a finite state space.  When $N$ is infinite, we call $(X_t,\La_t)$ a \rsp in an infinite countable state space.

To proceed, we introduce some conditions on the coefficients so that the solution of SDE (\ref{1.1}) (\ref{1.2}) exists and $0$ is the unique equilibrium point of this random dynamic system. Therefore, the stability studied in this work is mainly focused on whether the equilibrium $0$ is stable or not. Hence, it is natural to assume that the process is not explosive, which is ensured by the linear growth condition. In what follows, we introduce some conditions used in this work.
\begin{itemize}
  \item[(H2.1)]\ $(q_{ij}(x))$ is conservative and irreducible for each $x\in \R^d$. For each $x\in\R^d$, $i\in \S$, $m_{x,i}:=\sup\{j\in \S;\ q_{ij}(x)>0\}<\infty$.  There exists a constant $C$ such that
      $\sum_{j\neq i} j^2q_{ij}(x)\leq C(1+|x|^2)$ for every $x\in \R^d$, $i\in \S$. For each $n\in\N$, there exists a constant $C_n$ such that
      \[|q_{ij}(x_1)-q_{ij}(x_2) |\leq C_n|x_1-x_2|, \quad  x_1,\,x_2\in \R^d, \ |x_1|\leq n,\,|x_2|\leq n,\ i\neq j\in \S.
      \]
  \item[(H2.2 )]\ $b(0,i)=0$ and $\sigma(0,i)=0$ for each $i\in\S$. Moreover, for any sufficiently small $0<\veps<r_0$, there exist   $l\in\{1,\ldots,d\}$  and $\kappa(\veps)>0$ such that $a_{ll}(x,i)>\kappa(\veps)$ for all $(x,i)\in \{x; \veps<|x|<r_0\}\times \S$, where $a(x,i)=\sigma(x,i)\sigma(x,i)^\ast$.
\item[(H2.3)]\ For each $n\in \N$, there exists  constant $\bar K_n$ so that $|b(x,i)|+\|\sigma(x,i)\| \leq \bar K_n(1+|x|)$, for $x\in \R^d,\ |x|\leq n,\  i\in \S$.
\item[(H2.4)]\ For each $n\in\N$, there exists constant $\tilde K_n>0$ so that
\[|b(x,i)-b(y,i)|+\|\sigma(x,i)-\sigma(y,i)\|\leq \tilde K_n|x-y|,\quad \forall\, x,\,y\in \R^d, |x|\leq n, \, |y|\leq n,\ i\in \S.\]
\end{itemize}
Here and in the sequel, $\sigma^\ast$ stands for the transpose of matrix $\sigma$, and $\|\sigma\|$ denotes the operator norm. When $\S$ is a finite set, according to \cite{YZ},  conditions (H2.1), (H2.3) and (H2.4) ensure the existence of a nonexplosive solution $(X_t,\La_t)$ of (\ref{1.1}) and (\ref{1.2}). When $\S$ is an infinite countable set, according to the theory of SDE driven by L\'evy process, conditions (H2.1), (H2.3) and (H2.4) also ensure the existence of a nonexplosive solution of (\ref{1.1}) and (\ref{1.2}). Indeed, by \cite{GAM},  the jump process $(\La_t)$ can be represented by
a stochastic integral w.r.t. a Poisson random process. Precisely, for each $x\in \R^d$, $i,\,j\in\S$ with $i\neq j$, let $\Delta_{ij}(x)$ be a consecutive (w.r.t. the lexicographic ordering on $\S\times \S$), left-closed, right open intervals on the real line, each having length $q_{ij}(x)$. Namely,
\begin{align*}
  \Delta_{12}(x)&=\big[0,q_{12}(x)\big),\quad \Delta_{13}(x)=\big[q_{12}(x),q_{12}(x)+q_{13}(x)\big),\\
  &\quad \vdots\\
  \Delta_{1m_{x,1}}(x)&=\big[\sum_{j=2}^{m_{x,1}-1}q_{1j}(x),\sum_{j=2}^{m_{x,1}}q_{1j}(x)\big),\\
  \Delta_{21}(x)&=\big[\sum_{j=2}^{m_{x,1}}q_{1j}(x),\sum_{j=2}^{m_{x,1}}q_{1j}(x)+q_{21}(x)\big),
\end{align*}
and so on. Let $h:\R^d\times \S\times \R\ra \R$ be defined by
\[h(x,i,z)=\sum_{j=1}^{m_{x,i}}(j-i)\mathbf{1}_{\Delta_{ij}(x)}(z).\]
Then (\ref{1.2}) is equivalent to
\begin{equation}\label{1.2-1}
\d \La_t=\int_{\R}h(X_t,\La_{t-},z)N(\d t,\d z),
\end{equation}
where $N(\d t,\d z)$ is a Poisson random measure with intensity $\d t\times \d z$. Hence, by \cite[Section II-2.1, pp.104]{Sko}, conditions (H2.1), (H2.3) and (H2.4) ensure the existence of the solution of (\ref{1.1}) and (\ref{1.2}). Moreover, these conditions also ensure that
%
there exists a constant  $\tilde C $ such that
\[|b(x,y)|^2+\|\sigma(x,y)\|^2+\|h(x,y,\cdot)\|_{L^2(\d x)}\leq \tilde C(1+|x|^2+y^2), \ x\in \R^d,\,y\in \S.\]
According to \cite[Lemma 114]{Situ}, the solution of (\ref{1.1}) and (\ref{1.2}) is nonexplosive.

We adopt the definition of stability given in R. Khasminskii \cite{Kh}. See also
\cite[Chapter 7]{YZ}  or \cite{KZY2007}. Precisely, the equilibrium point $x=0$ is said to be \emph{stable in probability}, if for any $\veps>0$ and any $i\in \S$,
\[\lim_{x\ra 0}\p\big(\sup_{t\geq 0} |X^{x,i}(t)|>\veps\big)=0,\]
and $x=0$ is said to be \emph{unstable in probability} if it is not stable in probability. Here $(X^{x,i}(t))$ denotes the first component of the solution of (\ref{1.1}), (\ref{1.2}) with initial condition $(X_0,\La_0)=(x,i)$.  The equilibrium point $x=0$ is said to be \emph{asymptotically stable in probability}, if it is stable in probability and satisfies
\[\lim_{x\ra 0}\p\big(\lim_{t\ra \infty} X^{x,i}(t)=0\big)=1,\quad \text{for any $i\in \S$}.\]

For the convenience of reader, we collect some results from \cite{YZ} on Foster-Lyapunov criteria for \textbf{RSDP}.
Let $\mathscr A$ be the infinitesimal generator of $(X_t,\La_t)$ which is expressed by
\[\mathscr A f(x,i)=L^{(i)}f(\cdot,i)(x)+Qf(x,\cdot)(i),\]
where \[L^{(i)}=\frac{1}{2}\sum_{k,l=1}^da_{kl}(x,i)\frac{\partial^2}{\partial x_k\partial x_l}+\sum_{k=1}^db_k(x,i)\frac{\partial}{\partial x_k}, \  \text{and}\ \ Qg(i)=\sum_{j\neq i}q_{ij}(g_j-g_i).\]

Since the situation that $\S$ is an infinite countable set is considered, we need to extend the criteria established in \cite[Chapter 7]{YZ} for \rsp in a countable state space. The main idea is similar, but some new techniques are needed. For the ease of the reader, we give out these results.

\begin{lem}\label{lem-2.1}
Suppose that the conditions (H2.1) (H2.2) hold. Then
\[\p(X^{x,i}(t)\neq 0,\ t\geq 0)=1 \quad \text{for any $x\neq 0$, $i\in \S$},\]
and for any $p\in \R$, $t>0$,
\[\E\big[|X^{x,i}(t)|^p\big]\leq |x|^p e^{Kt}, \quad x\neq 0,\ i\in \S,\]
where $K$ is a constant depending only on $p$ and Lipschitz constant $\bar K$.
\end{lem}

\begin{proof}
When $\S$ is a countable set, this lemma can be proved in the same way as that of \cite[Lemma 7.1]{YZ}. But we should note that the constant $K$ could be taken independent of the cardinality of $\S$ in their proof.
\end{proof}

\begin{lem}\label{lem-2.2}
Assume that (H2.1) (H2.2) hold.
Let $D\subset \R^d$ be a neighborhood of $0$. Suppose that there exists a nonnegative function $V(\cdot,i):D\ra \R$ such that
\begin{itemize}
  \item[$\mathrm{(i)}$] $V(\cdot,i)$ is continuous in $D$ for each $i\in\S$ and $\inf_{i\in\S} V(x,i)$ vanishes only at $x=0$;
  \item[$\mathrm{(ii)}$] $V(\cdot,i)$ is twice continuously differentiable in $D\backslash \{0\}$, and satisfies $\mathscr A\,V(x,i)\leq 0$ for all $x\in D\backslash\{0\}$,  $i\in\S$.
\end{itemize}
Then the equilibrium point $x=0$ is asymptotically stable in probability.
\end{lem}
\begin{proof}
  This lemma can be proved in the same way as that of \cite[Lemma 7.5]{YZ}. There they used the condition that  $V(x,i)$ vanishes only at $0$ for each $i\in \S$. So when $N<\infty$ the condition (i) here is equivalent to the  condition in
  \cite[Lemma 7.5]{YZ}. But when $N=\infty$, we use condition (i) which ensures that the proof of \cite[Lemma 7.5]{YZ} is still valid. So the equilibrium point $x=0$ is stable in probability. By Remark 7.8 in \cite{YZ}, we can prove that the equilibrium point $x=0$ is asymptotically stable in probability in the same way as that of \cite[Lemma  7.6]{YZ}.
\end{proof}

\begin{lem}\label{lem-2.3}
Let $D\subset \R^d$ be a neighborhood of $0$. Suppose that (H2.1) and (H2.2) hold, and for each $i\in\S$, there exists a nonnegative function $V(\cdot,i):D\ra \R$ such that $V(\cdot,i)$ is twice continuously differentiable in $D\backslash \{0\}$. Suppose that
\begin{gather}\label{cond-1}
  \mathscr A\,V(x,i)\leq 0\quad \text{for all $x\in D\backslash\{0\}$, $i\in \S$},\\
  \label{cond-2}
  \lim_{x\ra 0}V(x,i)=\infty \quad \text{for each $i\in \S$}.
\end{gather}
Then the equilibrium point $x=0$ is unstable in probability.
\end{lem}

\begin{proof}
  Let $\zeta>0$ such that the closed ball $\bar B_{\zeta}=\{x;\,|x|\leq\zeta\}$ is contained in $D$, and $(x,i)\in B_\zeta\times \S$.
  Letting $0<\veps<|x|$ and $m\in \N$, $m\geq 2$, define
  \begin{align*}
    \tau_{\veps}&=\inf\{t\geq 0;\ |X^{x,i}(t)|\leq \veps\},\quad \tau_{\zeta}=\inf\{t\geq 0;\ |X^{x,i}(t)\geq \zeta\},\\
    \tau_{\veps,m}&=\inf\big\{t\geq0;\ |X^{x,i}(t)|\leq \veps,\ \La_t\in \{1,\ldots,m\}\big\}.
  \end{align*}
  Then it is obviously that $\tau_{\veps,m}\geq \tau_{\veps}$. By It\^o's formula,
  \begin{align*}\E\big[V(X(t\wedge\tau_{\zeta}\wedge\tau_{\veps,m}),
     \La(t\wedge\tau_{\zeta}\wedge\tau_{\veps,m}))\big]&= V(x,i)+\E\Big[\int_0^{t\wedge\tau_{\zeta}\wedge\tau_{\veps,m}}\mathscr AV(X_s,\La_s)\d s\Big]\\ &\leq V(x,i).
  \end{align*}
  Letting $t\ra\infty$, we get by Fatou's lemma
  \[\E\big[V(X(\tau_{\zeta}\wedge\tau_{\veps,m}),\La(\tau_{\zeta}\wedge\tau_{\veps,m}))\big]\leq V(x,i).\]
  As $V$ is nonnegative, we obtain
  \begin{equation}\label{ine-2}
  \begin{split}
  V(x,i)&\geq \E\big[V(X(\tau_{\veps,m}),\La(\tau_{\veps,m}))\mathbf 1_{\{\tau_{\veps,m}<\tau_{\zeta}\}}\big]\\
  &\geq \inf_{|y|\leq \veps,j\leq m}V(y,j)\p(\tau_{\zeta}>\tau_{\veps,m})\\
  &=\inf_{|y|\leq\veps,j\leq m} V(y,j)\p\big(\sup_{0\leq t\leq \tau_{\veps,m}}\big|X(t)\big|<\zeta\big).
  \end{split}
  \end{equation}
  Here we should note that at the time $\tau_{\veps,m}$, it is possible that $|X(\tau_{\veps,m})|<\veps$.
  By Lemma \ref{lem-2.1}, we have $\tau_\veps\ra \infty$ almost surely as $\veps\ra 0$. Indeed, set $A=\{\omega: \tau_0(\omega):=\lim_{\veps\ra 0}\tau_{\veps}(\omega)<\infty\}$. For $\omega\in A$, as $|X_{\tau_\veps}(\omega)|=\veps$, we get $|X_{\tau_0}(\omega)|=0$. If $\p(A)>0$, then $\p(X_t=0\ \text{for some}\ t\geq 0)\geq \p(A)>0$, which contradicts the result of Lemma \ref{lem-2.1}. Hence $\tau_{\veps,m}\ra \infty$ almost surely as $\veps\ra 0$.
  By (\ref{cond-2}), it holds $\lim_{\veps\ra 0}\inf_{|y|\leq \veps,j\leq m}V(y,j)=\infty$. Consequently, (\ref{ine-2}) yields
  \[\p\big(\sup_{t\geq 0}|X_t|\leq \zeta\big)=0,\]
  which shows that the equilibrium point $x=0$ is unstable in probability.
\end{proof}

\subsection{State-independent \rsp in a finite state space}
In this subsection, we mainly want to deal with the difficulty generated by the nonlinearity of the system, and leave the difficulty generated by the state-dependence and infiniteness of switching to the next subsection.
Therefore, in this subsection, $Q$-matrix $(q_{ij})$ is independent of $x$ and $\S$ is a finite set, i.e. $N<\infty$.

%
%
Now we introduce two conditions used later to characterize the stability of $(X_t)$ in each fixed environment. Let $D$ be a neighborhood of $0$. Let $\rho,\,h:D\ra[0,\infty)$ be nonnegative functions such that $\rho$, $h$ are twice continuously differentiable in $D\backslash\{0\}$.
\begin{itemize}
  \item[(A1)]  For each $i\in\S$, there exists a number $\beta_i\in \R$ such that
      \[L^{(i)}\rho(x)\leq \beta_ih(x),\quad \forall\,x\in D\backslash\{0\},\]
      and
      \[\lim_{x\ra 0}\frac{h(x)}{\rho(x)}=0,\quad \lim_{x\ra 0}\frac{L^{(i)}h(x)}{h(x)}=0.\]
  \item[(A2)] For each $i\in\S$, there exists a number $\bar{\beta}_i\in \R$ such that
      \[L^{(i)}\rho(x)\leq \bar\beta_i\rho(x),\quad \forall\, x\in D\backslash\{0\}.\]
\end{itemize}
Note that here $\beta_i$ and $\bar \beta_i$ are allowed to be positive or negative. When $\rho(x)$ vanishes only at $0$, the negativeness of $\beta_i$ and $\bar \beta_i$ ensures that the equilibrium point is stable in probability for the diffusion process associated with $(X_t)$ in the fixed environment $i$.

Next, let us introduce our first criterion for stability of \rsp based on  the Friedholm alternative, which can provide us some on-off type criteria for \rsp.
\begin{thm}\label{t-f}
Let $(X_t,\La_t)$ be a state-independent \rsp   satisfying (\ref{1.1}) (\ref{1.2}) with $N<\infty$. Assume (H2.1)-(H2.4) hold. Let $\rho,\,h\in C^2(D)$ be two nonnegative functions satisfying (A1). Let $\mu$ be the invariant probability measure of $(\La_t)$. Suppose that
\begin{equation}\label{con-1}
\sum_{i\in\S}\mu_i\beta_i<0.
\end{equation}
Then the equilibrium point $x=0$ is asymptotically stable in probability if $\rho(x)$ vanishes only at $0$, and is unstable in probability
if\, $\lim_{x\ra 0}\rho(x)=\infty$.
\end{thm}

\begin{proof}
  Since $\sum_{i=1}^N\mu_i\beta_i<0$, according to the Friedholm alternative (cf. \cite{PP}), there exists a constant $c>0$ and a vector $\xi$ such that
  \[Q\xi(i)=-c-\beta_i\quad \text{for every}\ i\in \S.\]
  Set $V(x,i)=\rho(x)+\xi_i h(x)$ with the vector $\xi=(\xi_1,\ldots,\xi_N)^\ast$  determined above.  We have
  \begin{equation}\label{ine-2.1}
  \begin{split}
    \mathscr A\,V(x,i)&=L^{(i)}\rho(x)+Q\xi(i)h(x)+\xi_iL^{(i)}h(x)\\
             &\leq\Big(\beta_i+Q\xi_i+\xi_i\frac{L^{(i)}h(x)}{h(x)}\Big)h(x), \quad x\in D\backslash\{0\}.
  \end{split}
  \end{equation}
  By (\ref{ine-2.1}), we obtain
  \begin{equation}
    \label{ine-2.2}
    \mathscr A\,V(x,i)\leq \Big(-c+\xi_i\frac{L^{(i)}h(x)}{h(x)}\Big) h(x),\quad  x\in D\backslash\{0\}.
  \end{equation}
  Since $(\xi_i)_{i=1}^N$ is bounded due to the finiteness of $N$, and $\dis\lim_{x\ra 0}\frac{L^{(i)}h(x)}{h(x)}=0$, for some $0<\veps<\frac c 2$, there exists $\delta_1>0$ such that
  for any $x\in D\cap\{z; 0<|z|<\delta_1\}$, one has $\big|\xi_i\frac{L^{(i)}h(x)}{h(x)}\big|<\veps<\frac c 2$, and hence
  $\mathscr A\,V(x,i)\leq - \frac c2 h(x)<0$. Now we check the nonnegativity of $V(x,i)$.
  By (A1), we have $\lim_{x\ra 0}1+\xi_i\frac{h(x)}{\rho(x)}=1$, so there exists $\delta_2>0$ such that for any $x\in \{z;|z|\leq \delta_2\}\cap D$,
  \[1+\frac{h(x)}{\rho(x)}\geq \frac 12,\quad \text{and}\ V(x,i)=\rho(x)\Big(1+\xi_i\frac{h(x)}{\rho(x)}\Big)\geq \frac 12 \rho(x)\geq 0.
  \]
  This also implies that if $\lim_{x\ra 0} \rho(x)=\infty$, then $\lim_{x\ra 0}V(x,i)=\infty$ for every $i\in \S$. Moreover, by the definition of $V(x,i)$ and the condition $\lim_{x\ra 0}\frac{h(x)}{\rho(x)}=0$, it is easy to see that if $\rho(x)$ vanishes only at $0$, then $V(0,i)=\lim_{x\ra 0}V(x,i)=0$ and $V(x,i)$ vanishes only at $0$. Set $\bar D=D\cap\{z;|z|<\min\{\delta_1,\delta_2\}\}$ being a neighborhood of $0$, then
  \[\mathscr A\,V(x,i)\leq 0,\quad \forall\, (x,i)\in \bar D\backslash\{0\}\times \S.\]
  Applying Lemma \ref{lem-2.2} and Lemma \ref{lem-2.3}, we can conclude the proof.
\end{proof}

Next, we give out a criterion based on the theory of M-matrix.
In \cite{MY}, the theory of M-matrix has been used to study the stability of linear state-independent regime-switching diffusion processes. In this work, we generalize this method to deal with more general regime-switching diffusion processes. As an application of this method, an example of nonlinear regime-switching diffusion process in a countable state space is given in subsection 2.2.

Let us first introduce some basic facts on the theory of M-matrix.
Let $B$ be a matrix or vector. By $B\geq 0$ we mean that all elements of $B$ are nonnegative. By $B>0$ we mean that $B\geq 0$ and at least one element of $B$ is positive. By $B\gg 0$, we mean that all elements of $B$ are positive. $B\ll 0$ means that $-B\gg 0$.
\begin{defn}[M-matrix] A square matrix $A=(a_{ij})_{n\times n}$ is called an M-Matrix if $A$ can be expressed in the form $A=sI-B$ with some $B\geq 0$ and $s\geq\mathrm{Ria}(B)$, where $I $ is the $n\times n$ identity matrix, and $\mathrm{Ria}(B)$ the spectral radius of $B$. When $s>\mathrm{Ria}(B)$,   $A$ is called a nonsingular M-matrix.
\end{defn}
Below, we cite some equivalent conditions that $A$ is a nonsingular M-matrix and refer the reader to \cite{BP} for more equivalent conditions.
\begin{prop}[\cite{BP},\cite{MY}]\label{m-matrix}
The following statements are equivalent.
\begin{enumerate}
\item $A$ is a nonsingular $n\times n$ M-matrix.
\item All of the principal minors of $A$ are positive; that is,
$\begin{vmatrix} a_{11}\!&\!\ldots\!&\!a_{1k}\\ \vdots& &\vdots\\ a_{1k}&\ldots&a_{kk}\end{vmatrix}>0 \ \  \text{for every $k=1, \ldots,n$}.$
\item Every real eigenvalue of $A$ is positive.
\item $A$ is semipositive; that is, there exists $x\gg 0$ in $\R^n$ such that $Ax\gg0$.
\end{enumerate}
\end{prop}

\begin{thm}\label{t-f-2}
 Let $(X_t,\La_t)$ be state-independent \rsp satisfying (\ref{1.1}) (\ref{1.2}) with $N<\infty$. Assume (H2.1)-(H2.4) hold. Let $\rho\in C^2(D)$ be a nonnegative function such that (A2) holds.  Suppose that
\begin{equation}
-\big(Q+\diag(\bar\beta_1,\ldots,\bar\beta_N)\big)\ \text{is a nonsingular M-matrix},
\end{equation}
where $\diag(\bar\beta_1,\ldots,\bar\beta_N)$ denotes the diagonal matrix generated by the vector $(\bar \beta_1,\ldots,\bar\beta_N)^\ast$ as usual.
Then the equilibrium point $x=0$ is asymptotically stable in probability if $\rho(x)$ vanishes only at $0$, and is unstable in probability
if\, $\lim_{x\ra 0}\rho(x)=\infty$.
\end{thm}

\begin{proof}
As $-(Q+\diag(\bar\beta_1,\ldots,\bar\beta_N))$ is a nonsingular M-matrix, there exists a vector $\xi=(\xi_1,\ldots,\xi_N)^\ast\gg 0$ such that
\[\lambda=(\lambda_1,\ldots,\lambda_N)^\ast:=-(Q+\diag(\bar\beta_1,\ldots,\bar\beta_N))\xi\gg 0.\]
Set $V(x,i)=\xi_i \rho(x)$ for $x\in D\backslash \{0\}$ and $i\in\S$, then
\begin{align*}
\mathscr A\,V(x,i)&=Q\xi(i)\rho(x)+\xi_iL^{(i)}\rho(x)\\
     &\leq (Q\xi(i)+\bar\beta_i\xi_i)\rho(x)=-\lambda_i\rho(x)\leq 0.
\end{align*}
Applying Lemma \ref{lem-2.2} and Lemma \ref{lem-2.3}, the desired conclusion follows immediately.
\end{proof}

Now we apply our criterion to study the stability of a class of nonlinear \textbf{RSDP}.  Let $(X_t,\La_t)$ satisfy (\ref{1.1}) and (\ref{1.2}). We assume further that
\begin{equation}\label{con-2}
b(x,i)=|x|^\gamma\hat b(x/|x|,i)(1+o(1)),\quad \sigma(x,i)=|x|^\zeta\hat\sigma(x/|x|,i)(1+o(1))
\end{equation}
as $x\ra 0$, where $\hat b(\cdot,\cdot):S^{d-1}\times \S\ra \R^d$, $\hat \sigma(\cdot,\cdot):S^{d-1}\times \S\ra \R^{d\times d}$, $\hat b$ and $\hat \sigma$ are continuous, $1<\gamma\leq 2\zeta-1$ and $S^{d-1}$ denotes the unit sphere in $\R^d$.
We define some quantities used later. Denote by $\theta=(\theta_1,\ldots,\theta_d)^\ast$ a point in $S^{d-1}$. Denote by $\delta_k(\cdot)$ the Dirac measure at $k$. Put $\hat a(\theta,i)=\hat\sigma(\theta,i)\hat\sigma(\theta,i)^\ast$.
For each $i\in \S$, set
\begin{equation}\label{beta-1}
\beta_i=\left\{\begin{array}{ll} \sup\limits_{\theta\in S^{d-1}}\sum_{k=1}^N \hat b_k(\theta,i)\theta_k,& \text{if}\ \gamma<2\zeta-1,\\
\sup\limits_{\theta\in S^{d-1}}\Big[\frac 12 \sum_{k,l=1}^N\hat a_{kl}(\theta,i)\big(\delta_{k}(l)-2\theta_k\theta_l\big)+\sum_{k=1}^N \hat b_k(\theta,i)\theta_k\Big],&\text{if}\ \gamma=2\zeta-1,\end{array}\right.
\end{equation}
and
\begin{equation}\label{beta-2}
\tilde \beta_i=\left\{\begin{array}{ll} \inf\limits_{\theta\in S^{d-1}}\sum_{k=1}^N \hat b_k(\theta,i)\theta_k,& \text{if}\ \gamma<2\zeta-1,\\
\inf\limits_{\theta\in S^{d-1}}\Big[\frac 12 \sum_{k,l=1}^N\hat a_{kl}(\theta,i)\big(\delta_{k}(l)-2\theta_k\theta_l\big)+\sum_{k=1}^N \hat b_k(\theta,i)\theta_k\Big],&\text{if}\ \gamma=2\zeta-1.\end{array}\right.
\end{equation}
\begin{thm}\label{t-nonlinear-1}
Assume (H2.1)-(H2.4) hold. Suppose that $(X_t,\La_t)$ satisfies (\ref{1.1}) (\ref{1.2}) with coefficients satisfying (\ref{con-2}) and $1<\gamma\leq 2\zeta-1$. Let $(\mu_i)$ be the invariant probability measure of $(\La_t)$.
The equilibrium point $x=0$ is asymptotically stable in probability if\, $\sum_{i=1}^N\mu_i \beta_i<0$, and is unstable in probability if\, $\sum_{i=1}^N\mu_i\tilde \beta_i>0$.
\end{thm}

\begin{proof}
 Let
  \[\tilde L^{(i)}=\frac12 |x|^{2\zeta}\sum_{k,l=1}^N\hat a_{kl}(x/|x|,i)\frac{\partial^2}{\partial x_k\partial x_l}+ |x|^\gamma\sum_{k=1}^N\hat b_k(x/|x|,i)\frac{\partial}{\partial x_k}.  \]

 (1)\ To prove that $x=0$ is stable in probability, we choose $\rho(x)= |x|^p \ (p>0)$ and $h(x)=|x|^{\gamma+p-1} $.
 Then
 \begin{align*}
   \tilde L^{(i)} \rho(x)&= \frac p2|x|^{2\zeta+p-2}\sum_{k,l=1}^N\hat a_{kl}(\theta,i)\big( (p-2) \theta_k\theta_l+\delta_k(l)\big) +p|x|^{\gamma+p-1}
   \Big[\sum_{k=1}^N\hat b_k(\theta,i)\theta_k\Big] \\
   &=p\Big[\frac{ |x|^{2\zeta-1-\gamma}}{2}\sum_{k,l=1}^N\hat a_{kl}\big(\theta,i\big)\big((p-2)\theta_k\theta_l +\delta_k(l) \big)+\sum_{k=1}^N\hat b_k(\theta,i)\theta_k\Big]h(x),
 \end{align*}
  where $\theta=(\theta_1,\ldots,\theta_N)^\ast=x/|x|$.
  By direct calculation, one gets
  \[\lim_{x\ra 0} \frac{h(x)}{\rho(x)}=0 \quad \text{and}\quad \lim_{x\ra 0}\frac{L^{(i)} h(x)}{h(x)}=0.\]
  If $\gamma<2\zeta-1$, then for any $\veps >0$ we can choose a $\delta_1>0$   so that for any $x\in\{z;|z|<\delta_1\}$,
  \[p\Big[\frac{ |x|^{2\zeta-1-\gamma}}{2}\sum_{k,l=1}^N\hat a_{kl}\big(\theta,i\big)\big((p-2)\theta_k\theta_l +\delta_k(l) \big)+\sum_{k=1}^N\hat   b_k(\theta,i)\theta_k\Big]\leq p\Big(\sum_{k=1}^N\hat b_k(\theta,i)\theta_k+\veps\Big).\]
  If $\gamma=2\zeta-1$, then
  for any $\veps>0$, there exists $p_0>0$ such that for any $0<p<p_0$,
  \[p\Big[\frac{ |x|^{2\zeta-1-\gamma}}{2}\sum_{k,l=1}^N\hat a_{kl}\big(\theta,i\big)\big((p-2)\theta_k\theta_l +\delta_k(l) \big)+\sum_{k=1}^N\hat   b_k(\theta,i)\theta_k\Big]\leq p(\beta_i+\veps).\]
  Invoking the condition (\ref{con-2}), by choosing a sufficiently small $\delta_1$ and $p_0$, we have
  \[L^{(i)}\rho(x)\leq p(\beta_i+\veps)h(x).\]
  By the arbitrariness of $\veps>0$, by Theorem \ref{t-f}, we obtain  that $x=0$ is asymptotically stable in probability if $\sum_{i=1}^N\mu_i\beta_i<0$.

  (2)\ We go to study the instability. Now let $\rho(x)=|x|^{-p}$ \ $(p>0)$ for $x\neq 0$ and $h(x)=|x|^{\gamma-p-1}$. Then it hold $\lim_{x\ra 0} \frac{h(x)}{\rho(x)}=0$ and $\lim_{x\ra 0}\frac{L^{(i)} h(x)}{h(x)}=0$. Analogous to the discussion in part (1), through choosing a small neighborhood of $0$ for $x$ or a small value $p$, we obtain for any $\veps>0$
  \begin{align*}
    \tilde L^{(i)}\rho(x)&=p\Big[\frac{|x|^{2\zeta-1-\gamma}}{2}\sum_{k,l=1}^N\hat a_{kl}(\theta,i)\big((p+2)\theta_k\theta_l-\delta_k(l)\big)-\sum_{k=1}^{N}\hat b_k(\theta,i)\theta_k\Big]h(x)\\
    &\leq -p(\tilde \beta_i+\veps)h(x).
  \end{align*}
  Therefore, if $\sum_{i=1}^N\mu_i\tilde \beta_i>0$, then $x=0$ is unstable in probability by Theorem \ref{t-f}.
\end{proof}

Theorem \ref{t-nonlinear-1} shows that our on-off criterion provided by Theorem \ref{t-f} could be very sharp. We can see it more clearly from the nonlinear systems in the 1-dimensional space.
Let
\begin{equation}\label{1-dim}
  \d X_t=b_{\La_t}\big(|X_t|^\gamma\wedge |X_t|\big)\d t+ \sigma_{\La_t}\big(|X_t|^\zeta \wedge |X_t|\big)\d B_t,  \quad \text{in}\ \R,
\end{equation} and $(\La_t)$ is still a continuous time Markov chain on $\S$ with $N<\infty$.
Applying Theorem \ref{t-nonlinear-1} to this situation, we can obtain immediately
\begin{cor}
  \label{cor-1} Suppose $(X_t,\La_t)$ satisfy  (\ref{1-dim}). Assume $1<\gamma\leq 2\zeta-1$.
  Let $\beta_i=b_i$ if $\gamma<2\zeta-1$ and $\beta_i=b_i-\frac 12 \sigma_i^2$ if $\gamma=2\zeta-1$. Then the equilibrium $x=0$ is asymptotically stable in probability if \,$\sum_{i=1}^N \mu_i\beta_i<0$ and is unstable in probability if \,$\sum_{i=1}^N\mu_i\beta_i>0$.
\end{cor}

\subsection{State-dependent \rsp in a countable state space}
In this subsection, we consider the stability of state-dependent \rsp in a countable state space. Based on our results for state-independent \rsp in a finite state space, we shall put forward a finite partition method to transform the \rsp in a countable state space into a new \rsp in a finite state space.  So in this subsection $(X_t,\La_t)$ still satisfies (\ref{1.1}) (\ref{1.2}) but with $N=\infty$ and $Q$-matrix $(q_{ij}(x))$ of $(\La_t)$ depending on $x$.

Let $\rho\in C^2(D)$ be a nonnegative function such that (A2) holds. As $N=\infty$, in this subsection we need to assume $M:=\sup_{i\in\S}\bar \beta_i<\infty$.
We first divide the space $\S$ into finite number of subsets according to the stability of $(X_t)$ in each fixed environment.
Precisely, let
\[\Gamma=\{-\infty=k_0<k_1<\ldots<k_{m-1}<k_m=M\}\]
be a finite partition of $(-\infty,M]$. Corresponding to $\Gamma$, there exists a finite partition $F=\{F_1,\ldots,F_m\}$ of $\S$ defined by
\begin{equation}
  \label{partition}
  F_i=\{j\in\S;\  \bar \beta_j\in (k_{i-1},k_i]\},\quad i=1,2,\ldots,m.
\end{equation}
We assume that each $F_i$ is nonempty, otherwise, we can delete some points in the partition $\Gamma$. Set
\begin{equation}
  \label{beta-F}\beta_i^F=\sup_{j\in F_i}\bar \beta_j,\quad q_{ii}^F=-\sum_{k\neq i}q_{ik}^F,
\end{equation}
\begin{equation}
  \label{q-F}
  q_{ik}^F=\left\{\begin{array}{ll}\sup_{x\in \R^d}\sup_{r\in F_i}\sum_{j\in F_k}q_{rj}(x), &\text{if}\ k<i,\\ \inf_{x\in\R^d}\inf_{r\in F_i}\sum_{j\in F_k}q_{rj}(x),&\text{if}\ k>i, \end{array}\right.
\end{equation} for $i,\,k\in \S$.
In order to ensure $q_{i}^F=-q_{ii}^F<\infty$ for $i=1,\ldots m$, we assume that there exists a number $\bar M$ such \begin{equation}\label{bound-q}
\sup_{x\in \R_d}\sup_{i\in \S} q_i(x)<\bar M<\infty.
\end{equation}
Then it is easy to see
\[\bar\beta_j\leq \beta_i^F,\ \forall\,j\in F_i,\ \text{and}\ \beta_{i-1}^F<\beta_i^F,\quad i=1,\ldots,m.\]
\begin{thm}\label{t-infi}
  \label{t-infi} Let $(X_t,\La_t)$ be a state-dependent \rsp in an infinite state space satisfying (\ref{1.1}) (\ref{1.2}). Assume that (H2.1)-(H2.4) and (\ref{bound-q}) hold. Let $\rho\in C^2(D)$ be a nonnegative function such that (A2) holds and $M=\sup_{i\in\S}\bar \beta_i<\infty$. Suppose that the $m\!\times\!m$ matrix $-\big(\diag(\beta_1^F,\ldots,\beta_m^F)+Q^F\big) H_m$ is a nonsingular M-matrix, where
  \begin{equation}\label{h-matrix}
H_m=\begin{pmatrix}
1&1&1&\cdots&1\\
0&1&1&\cdots&1\\
0&0&1&\cdots&1\\
\vdots&\vdots&\vdots&\cdots&\vdots\\
0&0&0&\cdots&1
\end{pmatrix}_{m\times m}.
\end{equation}
Then the equilibrium $x=0$ is asymptotically stable in probability if $\rho(x)$ vanishes only at $0$, and is unstable in probability if $\lim_{x\ra 0}\rho(x)=\infty$.
\end{thm}

\begin{proof}
Since $-(Q^F+\diag(\beta_1^F,\ldots,\beta_m^F))H_m$ is a nonsingular M-matrix, by Proposition \ref{m-matrix}, there exists a vector $\eta^F=(\eta_1^F,\ldots,\eta_m^F)^\ast\gg 0$ such that
\[\lambda^F=(\lambda_1^F,\ldots,\lambda_m^F)^\ast
:=-(Q^F+\diag(\beta_1^F,\ldots,\beta_m^F))H_m\eta^F\gg0.\]
Set $\xi^F=H_m\eta^F$. Then
\[\xi_i^F=\eta_m^F+\cdots+\eta_i^F,\quad i=1,\ldots,m,\]
which yields that $\xi^F_{i+1}<\xi_i^F$ for $i=1,\ldots,m-1$ and $\xi^F\gg 0$. We extend $\xi^F$ to a vector $\xi $ on $\S$ by setting $\xi_j=\xi_i^F$ if  $j\in F_i$.
Let $\phi:\S\ra \{1,\ldots,m\}$ be a map defined by $\phi(j)=k$ if $j\in F_k$. Let $Q_x g(i)=\sum_{j\neq i}q_{ij}(x)(g_j-g_i)$ for $g\in \mathscr B(\S)$. Set $V(x,i)=\xi_i\rho(x)$. By the definitions of $Q^F$, $\beta^F$ and the decreasing property of $\xi_i^F$,  we have, for $r\in F_i$,
\begin{align*}
Q_x\xi(r)&=\sum_{j\neq r} q_{rj}(x)(\xi_j-\xi_i)=\sum_{j\not\in F_i}q_{rj}(x)(\xi_j-\xi_i)\\
&=\sum_{k<i}\big(\sum_{j\in F_k}  q_{rj}(x)\big)(\xi_k^F-\xi_i^F)+\sum_{k>i}\big(\sum_{j\in F_k} q_{rj}(x)\big)(\xi_k^F-\xi_i^F)\\
&\leq \sum_{k<i}q_{ik}^F(\xi_k^F-\xi_i^F)+\sum_{k>i}q_{ik}^F(\xi_k^F-\xi_i^F) =Q^F\bxi^F(\phi(r)).
\end{align*}
Furthermore, \begin{align*}
\mathscr A\, V(x,r)&=Q_x\xi(r)\rho(x)+\xi_rL^{(r)}\rho(x)\\
&\leq \big(Q^F\xi^F(\phi(r))+\beta_{\phi(r)}^F\xi_{\phi(r)}^F\big)\rho(x)\\
&=-\lambda_{\phi(r)}\rho(x)\leq 0.
\end{align*}
Note that $\inf_{i\in \S}V(x,i)=\min_{i\leq m}\xi_i^F\rho(x)$ vanishes only at $0$, then
applying Lemmas \ref{lem-2.2} and \ref{lem-2.3}, we can get the desired conclusion.
\end{proof}

\begin{rem}
  The function $\rho(x)$ appeared in Theorem \ref{t-infi} is a test function, which we used to characterize  the behavior of the process $(X_t)$ in each fixed environment.
  Taking $\rho(x)$ to be a polynomial function will work for many cases as being shown by our examples.
  After getting the constants $\bar\beta_i$, $i\in \S$, one needs to classify the state space $\S$ according to $(\bar \beta_i)$. This part needs some skill. Other conditions of this theorem can be checked directly.
\end{rem}
\begin{exam}\label{exam-2.1}
Let $(X_t)$ be a process on $\R$ satisfying
\[\d X_t=b_{\La_t} X_t\d t+X_t^2\wedge |X_t|\, \d B_t,\quad X_0=x_0\neq 0, \]
and $(\La_t)$ is a birth-death process on $\S=\{1,2,\ldots\}$ with  $q_{ii+1}(x)=c_i+(i-1)\sin x$ for $i\geq 1$, $q_{ii-1}(x)=a_i+(i-2)\sin x$ for $i\geq 2$, $q_{ij}(x)=0$ for any $j\notin\{i-1,i,i+1\}$, where $a_i,\,c_i$ are positive constants. Let $(X_t^{(i)})$ be the diffusion process associated with $(X_t)$ in the fixed environment $i$, that is,
\[\d X_t^{(i)}=b_i X_t^{(i)}\d t+(X_t^{(i)})^2\wedge|X_t^{(i)}|\,\d B_t.\]
It is easy to know that $x=0$ is stable in probability for the process $(X_t^{(i)})$ if $b_i<0$, and is unstable in probability if $b_i>0$. By taking $\rho(x)=|x|$, it is easy to see $L^{(i)}\rho(x)=b_i|x|$ for $x\neq 0$. So we have $\bar\beta_i=b_i$ in condition (A2).
Take the partition as $F_1=\{1\}$ and $F_2=\{2,3,\ldots\}$. Then $q_{12}^F=c_1$, $q_{21}^F=a_2$.

More precisely,  we set $b_1=-1$ and $b_i=\kappa-i^{-1}$ for some positive constant $\kappa$.
Then $\beta_1^F=b_1=-1$ and $\beta_2^F=\kappa$. The matrix $-(Q^F+\diag(\beta_1^F,\beta_2^2))H_2$ is a nonsingular M-matrix if and only if $\kappa<\frac{a_2}{1+c_1}$.
Therefore, when $\kappa<\frac{a_2}{1+c_1}$, $x=0$ is asymptotically stable in probability for $(X_t,\La_t)$ according to Theorem \ref{t-infi}.
\end{exam}

\section{Principal eigenvalue and stability of \rsp}
In this section we continue to study the stability of the equilibrium point $x=0$ for the system given by (\ref{1.1}) and (\ref{1.2}). We shall provide some criteria in terms of the principal eigenvalue of a bilinear form. Let $(X_t,\La_t)$ be defined by (\ref{1.1}) and (\ref{1.2}) with state-independent $(q_{ij})$ and $N\leq \infty$. In this section, we need to assume that $(\La_t)$ is reversible. To emphasize on this condition, we use $(\pi_i)_{i\in \S}$ to denote the probability measure such that $\pi_iq_{ij}=\pi_jq_{ji}$, $i,\,j\in \S$.

Let $D$ be a neighborhood of 0. Let $\rho\in C^2(D)$ be a nonnegative function and satisfy
\begin{itemize}
  \item[(A3)] for each $i\in\S$, there exists a number $\gamma_i\in \R$ such that
  \[L^{(i)} \rho(x)\leq \gamma_i \rho(x),\quad |x|\in D\backslash\{0\},\]  and $\gamma_\cdot\not\equiv 0$.
\end{itemize}
Let $L^2(\pi)=\big\{f\in \mathscr B(\S); \sum_{i=1}^N\pi_i f_i^2<\infty\big\}$, and denote by $\|\cdot\|$ and $\la\cdot,\cdot\raa$ respectively the norm and inner product in $L^2(\pi)$.
Set
\begin{equation}\label{dirich}
\mathcal{E}(f)=\frac 12\sum_{i,j=1}^N\pi_iq_{ij}(f_j-f_i)^2-\sum_{i=1}^N\pi_i\gamma_if_i^2,\quad f\in L^2(\pi),
\end{equation}where $(\gamma_i)$ is given by (A3).
The domain of this bilinear form $\mathscr D(\mathcal E)$ is defined by
\[\mathscr D(\mathcal E)=\{f\in L^2(\pi);\ D(f)<\infty\}.\]
The principal eigenvalue of $\mathcal E(f)$ is defined by
\begin{equation}\label{prin-eig}
\lambda_0=\inf\big\{\mathcal{E}(f);\ f\in \mathscr D(\mathcal{E}),\ \|f\|=1\big\}.
\end{equation}
The notation $\mathcal{E}(f)$ is similar to the Dirichlet form for a Markov chain with killing, but in our case $\gamma_i$ could be positive which causes $\mathcal{E}(f)$ may take negative value for some $f\in L^2(\pi)$. Due to this fact, $\|f\|_{\mathcal{E}}^2:=\|f\|^2+\mathcal{E}(f)$ is no longer a norm. Some new difficulties appears when we provide a lower bound of $\lambda_0$ compared with the estimate of lower bound of the principal eigenvalue of a Dirichlet form.

Next, we apply the principal eigenvalue to study the stability of $(X_t,\La_t)$.
\begin{thm}\label{stab-1}
Suppose that (H2.1) (H2.2) hold and $N<\infty$. Assume that $(\La_t)$ is reversible w.r.t. the probability measure $(\pi_i)$.  Let $\rho\in C^2(D)$ be a nonnegative function such that (A3) holds. $\mathcal{E}(f)$ and $\lambda_0$ are defined by (\ref{dirich}) and (\ref{prin-eig}) as above.  Assume $\lambda_0>0$. Then
the equilibrium point $x=0$ is asymptotically stable in probability if $\rho(x)$ vanishes only at $0$ and is unstable in probability if $\lim_{x\ra 0}\rho(x)=\infty$.
\end{thm}

\begin{proof}
  Let \begin{equation}\label{omega}
  \Omega f(i)=Qf(i)+\gamma_i f_i=\sum_{j\neq i} q_{ij}(f_j-f_i)+\gamma_i f_i  \quad\text{for}\  f\in \mathscr D(D).\end{equation}
  Then $\mathcal{E}(f)=\la f,-\Omega f\raa$.
  When $N<\infty$, $\lambda_0$ must be the minimal eigenvalue of the operator $-\Omega$ by the variational representation of the eigenvalues of finite matrices.  Let $g$ be an eigenfunction corresponding to $\lambda_0$. Since $\mathcal{E}(g)\geq \mathcal{E}(|g|)$, it holds that $g\geq 0$. It is easy to see $g\not\equiv 0$. So there exists a $k\in \S$ such that
  $g_k>0$. If
  $q_{ik}>0$ for some $i\in \S$, then
  \[0<q_{ik} g_k\leq \sum_{j\neq i} q_{ij} g_j=(q_i-\gamma_i-\lambda_0)g_i.\]
  This yields that $g_i>0$ and $q_i-\gamma_i-\lambda_0>0$. As $Q$ is irreducible, we can prove inductively that $g_i>0$ for each $i\in \S$. Moreover, $\S$ is a finite set now, so $\min_{i\in \S} g_i>0$.

  Let $V(x,i)=g_i\rho(x)$ for $x\in D,\,i\in \S$. Due to (A3), we have
  \begin{align*}
  \mathscr A\,V(x,i)\leq \big(Q g(i)+\gamma_ig_i\big)\rho(x)&=-\lambda_0g_i\rho(x)\leq 0,\quad x\in D\backslash\{0\},\,i\in\S.
  \end{align*}
  If  $\lim_{x\ra 0}\rho(x)=\infty$, we have $\lim_{x\ra 0}V(x,i)=\lim_{x\ra 0}g_i\rho(x)=\infty$ for each $i\in \S$. Hence,  by Lemma \ref{lem-2.3}, the equilibrium point $x=0$ is unstable in probability.

  If $\rho(x)$ vanishes only at $0$, due to the finiteness of $N$, we have \[\inf_{i\in \S}V(x,i)=\big(\min_{i\in\S}g_i\big)\rho(x)\] also vanishes only at $0$. Applying Lemma \ref{lem-2.2}, we can get the desired result.
\end{proof}

\begin{thm}\label{stab-2}
Suppose that (H2.1) (H2.2) hold and $N=\infty$. Assume that $(\La_t)$ is reversible w.r.t. the probability measure $(\pi_i)$.  Let $\rho\in C^2(D)$ be a nonnegative function such that (A3) holds. $\mathcal{E}(f)$ and $\lambda_0$ are defined by (\ref{dirich}) and (\ref{prin-eig}) as above.  Assume $\lambda_0>0$  and $\lambda_0$ is attainable, i.e. there exists $g\in L^2(\pi)$, $g\not\equiv 0$, such that $\mathcal{E}(g)=\lambda_0\|g\|^2$. Then
\begin{itemize}
\item[$\mathrm{(i)}$] the equilibrium point $x=0$ is unstable in probability if $\lim_{x\ra 0}\rho(x)=\infty$.
\item[$\mathrm{(ii)}$] Assume further that $\liminf_{i\ra \infty} g_i\neq 0$, then the equilibrium point $x=0$ is asymptotically stable in probability if $\rho(x)$ vanishes only at $0$.
\end{itemize}
\end{thm}

\begin{proof}
We first show that $g\gg 0$ and $Qg(i)+\gamma_i g_i=-\lambda_0 g_i$, $i\in \S$.
To obtain this, we use the variational method as in \cite{Ch00}. It is easy to check $\mathcal{E}(f)\geq \mathcal{E}(|f|)$, so it must hold $g\geq 0$. For a fixed $k\in\S$, let $\tilde  g_i=g_i$ for $i\neq k$ and $\tilde g_k=g_k+\veps$.
  It holds $Q\tilde g(i)=Qg(i)+\veps q_{ik}$ for $i\neq k$ and $Q\tilde g(k)=Qg(k)-\veps q_k$.
We have
\begin{align*}
\mathcal{E}(\tilde g)&=\la \tilde g,-Q\tilde g\raa-\sum_{i=1}^N\pi_i\gamma_i\tilde g_i^2\\
&=\la  g, -Qg\raa-\sum_{i=1}^N\pi_i \gamma_i  g_i^2 +2\veps \pi_k(-Qg)(k) -2\veps\pi_k\gamma_kg_k+\veps^2\pi_k(q_k-\gamma_k),
\end{align*}where we have used the fact $\pi_iq_{ik}=\pi_kq_{ki}$.
Because $\mathcal{E}(\tilde g)\geq \lambda_0\|\tilde g\|^2$ and $\mathcal{E}(g)=\lambda_0\|g\|^2$, we get
\begin{equation}\label{eigen}
-2\veps \pi_k\big(\lambda_0g_k+Qg(k)+\gamma_kg_k\big)+\veps^2\pi_k(q_k-\gamma_k)-\lambda_0\veps^2\pi_k\geq 0.
\end{equation}
Dividing both sides of (\ref{eigen}) by $\veps>0$ and then letting $\veps\ra 0^+$, we get
$Qg(k)+\gamma_kg_k+\lambda_0 g_k\leq 0$. Dividing both sides of (\ref{eigen}) by $\veps<0$ and then letting $\veps\ra 0^-$, we get $Qg(k)+\gamma_k g_k+\lambda_0 g_k\geq 0$. Therefore, $Qg(k)+\gamma_k g_k=-\lambda_0 g_k$. Since $k$ is chosen arbitrarily in $\S$, we have
$Qg(i)+\gamma_i g_i=-\lambda_0 g_i$ for each $i\in \S$.

Since $g\not\equiv 0$ and $g\geq 0$, there exists $k$ such that $g_k>0$. If $q_{ik}>0$, then
\[0<q_{ik}g_k\leq \sum_{j\neq i}q_{ij}g_j=(q_i-\gamma_i-\lambda_0)g_i,\]
so $g_i>0$ and $q_i-\gamma_i-\lambda_0>0$. As $Q$ is irreducible, by an inductive procedure, we can prove that $g_i>0$ for every $i\in \S$.

Next, set $V(x,i)=g_i \rho(x)$ for $x\in\R^d$, $i\in \S$. We obtain for $|x|\geq r_0$, $i\in \S$,
\begin{equation}\label{ine-lya-2}
\begin{split}
  \mathscr A\, V(x,i)&=Qg(i)\rho(x)+g_i L^{(i)} \rho(x)\\
  &\leq \big(Qg(i)+\gamma_ig_i\big)\rho(x)=-\lambda_0 g_i\rho(x)\leq 0.
\end{split}
\end{equation}
If $\lim_{x\ra 0}\rho(x)=\infty$, then for each $i\in \S$, we have $\lim_{x\ra 0}V(x,i)=\lim_{x\ra 0}g_i\rho(x)=\infty$. By Lemma \ref{lem-2.3}, we can prove the assertion (i).
Under the condition $\liminf_{i\ra \infty} g_i\neq 0$, we get that  $\inf_{i\in\S}g_i>0$. Therefore, when $\rho(x)$ vanishes only at $0$, we can obtain that $\inf_{i\in\S}V(x,i)=\big(\inf_{i\in \S} g_i\big)\rho(x)$ still vanishes only at $0$. Then the assertion (ii) follows from Lemma \ref{lem-2.2}, and we complete the proof.
\end{proof}

Using the principal eigenvalue to study the stability of $(X_t,\La_t)$ when $N=\infty$, we don't need to assume the boundedness of $Q$-matrix and the vector $(\gamma_i)$ given by (A3). We construct an example of nonlinear \rsp to show its usefulness. We shall give a lower bound estimate of $\lambda_0$ in next section.

\begin{exam}\label{exam-3}
Let $(X_t)$ be a process on $\R$ satisfying
\begin{equation}\label{proc-2}
\d X_t=\mu_{\La_t}X_t\d t+X_t^2\wedge |X_t|\d B_t, \quad X_0=x_0\neq 0,
\end{equation}
where $a\wedge b=\min\{a,b\}$, and $\mu_0=-c$, $\mu_i=\gamma$ for $i\geq 1$, $c,\gamma$ are positive constants.  $(\La_t)$ is a birth-death process on $\S=\{0,1,\ldots\}$ with $q_{ii+1}=b_i=b(i+1)$, $q_{ii-1}=a_i=a(i+1)$. Taking $\rho(x)=|x|$, then $L^{(i)}\rho(x)=\mu_i\rho(x)$ for $x\neq 0$. Set $g_i=i+1$. If $a-b-\gamma>0$ and $c-b>0$, then there exists $\lambda>0$ such that $\Omega g(i)\leq -\lambda g_i$, $i\geq 0$.  So $\lambda_0\geq \lambda>0$ by Theorem \ref{t-var} below. Noting $\inf_{i\in\S} g_i=1>0$, by Theorem \ref{stab-2}, the equilibrium point $x=0$ is asymptotic stable in probability if $a-b-\gamma>0$ and $c-b>0$.
\end{exam}

\section{Principal eigenvalue and recurrence of \rsp}
In this section, we shall provide some criteria for the recurrence of \rsp in terms of the principal eigenvalue of a bilinear form. In \cite{Sh14b}, we have provided some explicit criteria for the recurrence of \textbf{RSDP}. But there when dealing with switching in a countable state space, we need to assume the boundedness of $Q$-matrix. In this part, we shall take advantage of the principal eigenvalue to provide a criterion for recurrence of \rsp without the boundedness assumption on the $Q$-matrix.

Let $(X_t,\La_t)$ satisfy (\ref{1.1}) (\ref{1.2}) with $(q_{ij})$ being state-independent. Assume $(\La_t)$ is reversible w.r.t. the probability measure $(\pi_i)_{i\in\S}$.
We use the following condition on the coefficients of $(X_t,\La_t)$ in this section.
\begin{itemize}
\item[(H4.1)]\ There exists a constant $K$ such that
  \begin{gather*}
    |b(x,i)|+\|\sigma(x,i)\|\leq K(1+|x|), \ x\in\R^d, \ i\in\S, \\
    |b(x,i)-b(y,i)|+\|\sigma(x,i)-\sigma(y,i)\|\leq K |x-y|,\ x\in \R^d,\ i\in\S,
  \end{gather*}
\end{itemize}
Let $\rho\in C^2(\R^d)$ be a positive function satisfying
\begin{itemize}
  \item[(A4)] for each $i\in \S$, there exists a number $\gamma_i\in \R$ such that
  \[L^{(i)} \rho(x)\leq \gamma_i\rho(x),\quad |x|\geq r_0\] for some $r_0>0$. Moreover, $\gamma_\cdot\not\equiv 0$.
\end{itemize}
We can define the bilinear form $\mathcal{E}(f)$ and it principal eigenvalue $\lambda_0$ by (\ref{dirich}) and (\ref{prin-eig}) similarly by using $(\gamma_i)$ given by (A4).
\begin{thm}\label{t-4.1}
  Suppose that (H4.1) holds and $N<\infty$. Suppose that $(\La_t)$ is reversible w.r.t. a probability measure $(\pi_i)$. Let $\rho\in C^2(\R^d)$ satisfy $\rho>0$ and (A4). Assume $\lambda_0>0$.
  Then $(X_t,\La_t)$ is positive recurrent if\, $\lim_{|x|\ra \infty} \rho(x)=\infty$ and is transient if\, $\lim_{|x|\ra \infty} \rho(x)=0$.
\end{thm}

\begin{proof}
  When $N<\infty$, $\lambda_0$ must be the minimal eigenvalue of the operator $-\Omega$. Let $g$ be an eigenfunction corresponding to $\lambda_0$. Then $g$ must be positive, i.e. for each $i\in \S$, $g_i>0$ and hence $\min_{i\in\S} g_i>0$ due to the finiteness of $N$. This fact follows from the irreducibility of $(q_{ij})$ and the minimal property of $\lambda_0$. Let $V(x,i)=g_i\rho(x)$ for $x\in \R^d$ and $i\in \S$. By (A1), we have
  \begin{equation}\label{ine-lya}
  \begin{split}
  \mathscr A\,V(x,i)&=Qg(i)\rho(x)+g_iL^{(i)}\rho(x)\\
  &\leq \big(Qg(i)+\gamma_i g_i\big)\rho(x)=-\lambda_0 g_i\rho(x).
  \end{split}
  \end{equation}
Therefore, if $\lim_{|x|\ra \infty}\rho(x)=\infty$, there exist positive constants $\veps$ and $r_0>0$ such that $g_i\rho(x)\geq \veps>0$ for every $|x|\geq r_0$, $i\in \S$. Applying the method of Lyapunov function (see \cite[Section 2]{PP}), $(X_t,\La_t)$ is positive recurrent. If $\lim_{|x|\ra \infty}\rho(x)=0$, then $(X_t,\La_t)$ is transient due to inequality (\ref{ine-lya}) and the method of Lyapunov function.
\end{proof}

\begin{thm}\label{t-4.2}
Suppose that (H4.1) holds and $N=\infty$. Suppose $(\La_t)$ is reversible w.r.t. a probability measure $(\pi_i)$. Let $\rho\in C^2(\R^d)$ be a positive function such that (A4) holds. Assume that    $\lambda_0>0$ and  $\lambda_0$ is attainable, i.e. there exists $g\in L^2(\pi)$, $g\not\equiv 0$ such that $\mathcal{E}(g)=\lambda_0\|g\|^2$. Then
$(X_t,\La_t)$ is transient if\, $\lim_{|x|\ra \infty} \rho(x)=0$. Assume further that $\liminf_{i\ra \infty} g_i\neq 0$, then $(X_t,\La_t)$ is recurrent if\, $\lim_{|x|\ra \infty}\rho(x)=\infty$.
\end{thm}

\begin{proof}
  Similar to the argument of Theorem \ref{stab-2}, it holds that for each $i\in\S$, $\Omega g(i)=-\lambda_0 g_i$ and $g_i>0$. Let $V(x,i)=g_i\rho(x)$ for $x\in \R^d$, $i\in \S$. We have for $|x|\geq r_0$, $i\in\S$,
  \begin{equation*}
    \mathscr A V(x,i)\leq \big(Q g(i)+\gamma_i g_i\big) \rho(x)=-\lambda_0g_i\rho(x)\leq 0.
  \end{equation*}
Let
\[\tau=\inf\big\{t>0;\ (X_t,\La_t)\in \{x\in \R^d; |x|\leq r_0\}\times \{1,\ldots,m_0\}\big\},\]
where $m_0$ is a fixed finite number. Applying It\^o's formula to $(X_t,\La_t)$ with $X_0=x$, $\La_0=l$, and $|x|>r_0$, $l>m_0$, we obtain
\begin{equation}\label{Ito}
\E V(X_{t\wedge \tau}, \La_{t\wedge \tau})=V(x,l)+\E\int_0^{t\wedge \tau} \mathscr A\, V(X_s,\La_s)\d s\leq V(x,l)=g_l\rho(x).
\end{equation}
Case 1. When  $\lim_{|x|\ra \infty} \rho(x)=0$, if $\p(\tau<\infty)=1$, then passing $t$ to $\infty$ in (\ref{Ito}), we obtain
\begin{equation}\label{ine-1}\inf_{\{y;|y|\leq r_0\}}\rho(y)\min_{i\leq m_0} g_i\leq \E\big[ \rho(X_\tau)g(\La_\tau)\big]\leq g_l\rho(x).
\end{equation}
Since the set $\{y;|y|\leq r_0\}\times \{1,\ldots,m_0\}$ is compact and the functions $\rho$ and $g$ are positive, the left-hand side of (\ref{ine-1}) is strictly positive. Since $x$ is arbitrary, letting $|x|\ra \infty$, the right-hand side of (\ref{ine-1}) tends to 0, which is a contradiction. Therefore, $\p(\tau<\infty)>0$ and $(X_t,\La_t)$ is transient.

Case 2. When $\lim_{|x|\ra\infty}\rho(x)=\infty$, we consider another stopping time
\[\tau_K=\inf\big\{t>0;\ |X_t|\geq K\big\}.\]
As the process $(X_t,\La_t)$ is nonexplosive, $\tau_K$ increases to $\infty$ a.s. as $K\ra \infty$. 
Since $\liminf_{i\ra\infty} g_i\neq 0$ and we have proved that $g>0$, we get $\inf_{i\in\S} g_i>0$. By It\^o's  formula, we get
\[\E \big[g(\La_{t\wedge\tau_K\wedge\tau})\rho(X_{t\wedge\tau_K\wedge\tau})\big]\leq \rho(x)g_l.\]
This yields
\begin{gather*}
\p(\tau>\tau_K)\leq \frac{\rho(x)g_l}{\inf_{\{y;|y|\geq K\}}\rho(y)\inf_{i\in \S} g_i}.
\end{gather*}
Since $\lim_{|x|\ra \infty }\rho(x)=\infty$, we obtain that  $\inf_{\{y;|y|\geq K\}} \rho(y)\ra \infty$ as $K\ra \infty$ and
\begin{equation*}
  \p(\tau=\infty)\leq \lim_{K\ra \infty}\frac{\rho(x)g_l}{\inf_{\{y;|y|\geq K\}}\rho(y)\inf_{i\in \S} g_i}=0.
\end{equation*}
Therefore, $\p(\tau<\infty)=1$ and the process $(X_t,\La_t)$ is recurrent.
\end{proof}

To apply our criteria, it is necessary to check the condition $\lambda_0>0$. There are many works and methods devoted to the estimate of the eigenvalues of Dirichlet forms. Next, we generalize the lower estimate of the principal eigenvalue of a Dirichlet form given in \cite{Ch00}. The usefulness of this lower estimate will be seen from the examples given in the end of this section. Following the approach of \cite{Ch00}, we first establish a variational formula for the bilinear form $\mathcal{E}(f)$ defined by (\ref{dirich}).

\begin{prop}\label{prop-1}
  Assume that $(q_i)_{i\in \S}$ and $(|\gamma_i|)_{i\in\S}$ are bounded. Let $f\in L^2(\pi)$ be nonnegative, then
  \[\mathcal{E}(f)=\sup_{g}\la f^2/g,-\Omega g\raa,\]
  where the supremum takes over all measurable functions $g$ which are strictly positive (i.e. $g\geq c_g>0$ for some constant $c_g$).
  Furthermore, if there exists strictly positive $g$ such that $\Omega g\leq -\lambda g$ for some $\lambda>0$, then $\lambda_0\geq \lambda$.
\end{prop}

\begin{proof}
  1)\ Let $\tilde g=g\mathbf 1_{f\neq 0}$, then $\la f^2/g,-\Omega g\raa \leq \la f^2/\tilde g,-\Omega g\raa$, which can be seen from the formula below.
  \begin{align*}
    \la f^2/g,-\Omega g\raa&=\sum_{i} \pi_i \frac{f_i^2}{g_i}\big(\sum_{j\neq i} q_{ij} (g_i-g_j)-\gamma_i g_i\big)\\
    &=\sum_{i}\sum_{j\neq i} \pi_i q_{ij} \frac{f_i^2}{g_i}(g_i-g_j)-\sum_{i}  \pi_i \gamma_i f_i^2\\
    &=\sum_{i} \pi_i q_i f_i^2-\sum_{i}\sum_{j\neq i} \pi_i q_{ij} \frac{f_i^2}{g_i} g_j-\sum_{i} \pi_i\gamma_if_i^2.
  \end{align*}
  We should verify whether the minus of  two infinite series is of meaning.
  Let $h_i=f_i/\tilde g_i$ then $h_i>0$.
  The previous equation
  \begin{align*}
    &=\sum_{i}\pi_ig_if_i^2-\sum_{i}\sum_{j\neq i}  \pi_i q_{ij} f_i h_i\frac{f_j}{h_j}-\sum_{i}\pi_i\gamma_if_i^2\\
    &=\sum_{i}\pi_iq_i f_i^2-\frac 12 \sum_{i}\sum_{j\neq i} \pi_i q_{ij}f_if_j\Big(\frac{h_j}{h_i}+\frac{h_i}{h_j}\Big)-\sum_{i}\pi_i\gamma_i f_i^2\\
    &\leq \sum_{i} \pi_iq_i f_i^2 -\sum_i\sum_{j\neq i}\pi_iq_{ij} f_if_j-\sum_{i}\pi_i\gamma_if_i^2\\
    &=\mathcal{E}(f).
  \end{align*}
  2) For the converse, if $f>c_f>0$, then it follows directly by taking $g=f$. If not, we let $f_n=n^{-1}+f$, then
  \begin{align*}
    \la f^2/f_n,-\Omega f_n\raa &=\frac 12\sum_{i,j}\pi_i q_{ij}\Big(f_j^2+f_i^2-\frac{f_j^2 f_n(j)}{f_n(i)}-\frac{f_j^2 f_n(i)}{f_n(j)}\Big)-\sum_{i}\pi_i\gamma_i f_i^2\\
    &=\frac12 \sum_{i,j}\pi_iq_{ij} \Big(\frac{f_j^2}{f_n(j)}-\frac{f_i^2}{f_n(i)}\Big)\big(f_n(j)-f_n(i)\big)-\sum_{i}\pi_i\gamma_i f_i^2.
  \end{align*}
  Since \[\Big(\frac{f_j^2}{f_n(j)}-\frac{f_i^2}{f_n(i)}\Big)\big(f_n(j)-f_n(i)\big)\geq 0,\] by Fatou's lemma, $\mathcal{E}(f)\leq \liminf_{n\ra \infty} \la f^2/f_n,-\Omega f_n\raa\leq \sup_g\la f^2/g,-\Omega g\raa$.

  The last statement follows directly from the variational formula for $\mathcal{E}(f)$.
\end{proof}

To deal with the general $(q_{ij})$ and $(\gamma_i)$, we need use the localization method to approximate the general situation. As $\|f\|_{\mathcal{E}}$ is no longer a norm, the triangular inequality does not hold and so we have to approximate $\mathcal{E}(f)$ directly.
\begin{thm}\label{t-var}
  If there exists a measurable function $g$ satisfying $g>0$ and $-\Omega g/g\geq \lambda>0$, then $\lambda_0\geq \lambda$.
\end{thm}

\begin{proof}
  For $n,\,m\in \N$, let
  \begin{gather*}
  G_{nm}=\{i;\ g_i\geq 1/m, |\gamma_i|\vee q_i\leq n\},\\ \lambda_{nm}=\inf\{\mathcal{E}(f);\ f\big|_{G_{nm}^c}=0,\,f\in \mathscr D(\mathcal{E})\}.
  \end{gather*}
  Let \[\mathcal{E}_{nm}(f)=\frac 12 \sum_{i,j\in G_{nm}} \pi_i q_{ij}(f_j-f_i)^2+\sum_{i\in G_{nm}}\pi_i \big(q_i-q(i,G_{nm})\big) f_i^2-\sum_{i\in G_{nm}} \pi_i\gamma_if_i^2,\]
  for $f\in L^2(G_{nm},\pi)$.
  Then it holds that
  for $f\big|_{G_{nm}^c}=0$, $\mathcal{E}(f)=\mathcal{E}_{nm}(f)$, and hence \[\lambda_{nm}=\inf\{\mathcal{E}_{nm}(f);\|f\big|_{G_{nm}^c} \|=1,\ f\in \mathscr D(\mathcal{E})\big\}.\]
  By this formula and the definition of $\lambda_0$, we get   \[\lambda_{nm}\geq \lambda_0\quad \text{and hence}\ \lim_{n,m\ra\infty}\lambda_{nm}\geq \lambda_0.\]
  For the converse, we have for any $\veps>0$ there exists $f_\veps\in \mathscr D(\mathcal{E})$ with $\|f_\veps\|=1$ such that
  $\lambda_0\geq \mathcal{E}(f_\veps)-\veps$. Let $f_{nm}=f_\veps\mathbf 1_{G_{nm}}$. We shall show  $\mathcal{E}(f_{nm})\ra \mathcal{E}(f_\veps)$. Indeed,
  \begin{align*}
    \mathcal{E}(f_{nm})&=\frac 12\sum_{i,j}\pi_iq_{ij}(f_{nm}(j)-f_{nm}(i))^2-\sum_{i}\pi_i\gamma_if_{nm}(i)^2\\
    &=\frac 12\sum_{i,j\in G_{nm}} \pi_iq_{ij}(f_j-f_i)^2+\sum_{i}\pi_i q(i, G_{nm}^c)f_i^2-\sum_{i\in G_{nm}} \pi_i\gamma_if_i^2.
  \end{align*}
  We have
  \begin{align*}
    &\big|\sum_{i} \pi_iq(i,G_{nm}^c)f_i^2+\sum_{i\in G_{nm}^c}\pi_i\gamma_i f_i^2\big|\\
    &\leq \sum_{i}\pi_iq_i f_i^2\frac{q(i,G_{nm}^c)}{q_i}+\sum_{i\in G_{nm}^c}\pi_i|\gamma_i|f_i^2\ra 0,  \end{align*} as $n,\,m\ra \infty$.  So $\lim_{n,m\ra \infty}\mathcal{E}(f_{nm})=\mathcal{E}(f)$. This yields that $\lambda_0\geq \lim_{n,m\ra\infty} \lambda_{nm}-\veps$, and hence $\lambda_0\geq \lim_{n,m\ra\infty}\lambda_{nm}$ by the arbitrariness of $\veps$.

  To complete the proof, we also need the following local operator. Let $B$ be a measurable set, define for $i\in B$
  \begin{gather*}
    q^B_i=q_i,\ q_{ij}^B=q_{ij}\mathbf 1_B(j),\\
    \Omega^B f(i)=\sum_{j\in B} q_{ij}^B f(j)-q_i^B f(i)+\gamma_i f(i).
  \end{gather*}
  If $-\Omega g\geq \lambda g$ then $-\Omega^B g\geq \lambda g$ on $B$ for every measurable set $B$. Indeed,
  \begin{align*}
    \Omega^Bg(i)&=\sum_jq_{ij}^Bg_j-q_i^B g_i+\gamma_i g_i\\
    &\leq \sum_j q_{ij} g_j-q_i g_i +\gamma_i g_i=\Omega g(i)\leq -\lambda g_i,\ i\in B.
  \end{align*}
  Applying this fact to the set $G_{nm}$, we get $\Omega^{G_{nm}} g\leq -\lambda g$ on $G_{nm}$. But on $G_{nm}$, $q_i$ and $|\gamma_i|$ are bounded and $g$ is strictly positive, by Proposition \ref{prop-1}, it holds $\lambda_{nm}\geq \lambda$. As $\lim_{n,m\ra \infty} \lambda_{nm}=\lambda_0$, we get $\lambda_0\geq \lambda$ and the proof is complete.
\end{proof}

\begin{rem}\label{rem-var}
We go back to see the argument of Theorem \ref{t-4.2}. If we can find a measurable $g$ on $\S$ such that $g>0$ and $\Omega g\leq -\lambda g$ for some $\lambda>0$, then we can define $V(x,i)=g_i \rho(x)$ by this $g$. The inequality (\ref{ine-lya-2}) holds with $\lambda_0$ replaced by $\lambda$. Following the proof of Theorem \ref{t-4.2}, we can get that $(X_t,\La_t)$ is transient if $\lim_{|x|\ra \infty}\rho(x)=0$. Under further condition on $g$, we get $(X_t,\La_t)$ is recurrent if $\lim_{|x|\ra \infty}\rho(x)=\infty$. The reason that we use $\lambda_0$ in Theorem \ref{t-4.2} is that $\lambda_0$ is the largest desired one by Theorem \ref{t-var}.
\end{rem}

At last, we give two examples of \rsp with switching in a countable state space to show the usefulness of our criteria. The $Q$-matrices of these two examples are unbounded. Especially, in the second example,  the vector $(\gamma_i)$ given by (A4) is also unbounded.
\begin{exam}\label{exam-1}
  Let $(X_t)$ satisfy the following stochastic differential equation,
  \begin{equation}\label{proc-1}
  \d X_t= \mu_{\La_t}X_t\d t+\d B_t,\quad X_0=x_0\in \R,
  \end{equation} where $(B_t)$ is a Brownian motion on $\R$ and $(\La_t)$ is a birth-death process on $\S=\{0,1,\ldots\}$. $(\La_t)$ and $(B_t)$ are mutually independent. Assume that $\mu_0=-c$, $\mu_i=\gamma$ for $i\geq 1$, where $c, \,\gamma$ are positive constants. Therefor, $(X_t^{(0)})$ is recurrent and $(X_t^{(i)})$ ($i\geq 1$) is transient. Assume $q_{i i+1}=b_i=b(i+1)$ for $i\geq 0$, and $q_{ii-1}=a_i=a(i+1)$ for $i\geq 1$. Here the constants $a$ and $b$ satisfy $a>b>0$, which ensures that $(\La_t)$ is reversible with respect to a probability measure $(\pi_i)$ (cf. Van Doorn \cite{Van}). The $Q$-matrix of birth-death process $(\La_t)$ is obviously unbounded.
  Then the process $(X_t,\La_t)$ is recurrent if $c-b>0$ and $a-b-\gamma>0$.

  Indeed, it is easy to see that for each $i\in\S$, \[L^{(i)}=\frac{1}2\frac{\d^2 }{\d x^2} +\mu_ix\frac{\d }{\d x}.\] Take $\rho(x)=|x|$, then it holds
  \[L^{(i)} \rho(x)=\mu_i\rho(x)\quad \text{for}\ |x|\geq 1,\ i\in \S.\]
  Then we can take $\gamma_i=\mu_i$ for $i\geq 0$ in the condition (A1). Define the bilinear form $\mathcal{E}(f)$ and its principal eigenvalue $\lambda_0$ as (\ref{dirich}) and (\ref{prin-eig}).
  If $c-b>0$ and $a-b-\gamma>0$, there exists a number $\lambda>0$ such that
  \[c-b\geq \lambda>0\ \ \text{and}\ \ a-b-\gamma\geq \lambda>0.\]
  Let $g_i=i+1$ for $i\geq 0$.  We have
  \begin{align*}
    \Omega g(i)&=b_i(g_{i+1}-g_i)+a_i(g_{i-1}-g_i)+\mu_ig_i\\
               &=b(i+1)-a(i+1)+\gamma (i+1)\\ &\leq -\lambda(i+1)=-\lambda g_i, \qquad \text{for}\ i\geq 1,\\
    \Omega g(0)&=b_0(g_1-g_0)+\mu_0 g_0=(b-c)g_0\leq -\lambda g_0.
  \end{align*}
  Applying Theorem \ref{t-var}, we get $\lambda_0\geq \lambda>0$. Combining with the fact $\lim_{|x|\ra \infty}\rho(x)=\infty$, it follows immediately from Theorem \ref{t-4.2} and Remark \ref{rem-var} that $(X_t,\La_t)$  is recurrent.
\end{exam}

\begin{exam}\label{exam-2}
  Let $(X_t)$   satisfy (\ref{proc-1}) with $\mu_0<-1$ and $\mu_i=i$. $(\La_t)$ is a birth-death with birth rate $b_i=1$ for all $i\geq 0$ and death rate $a_i>1+\mu_i(i+1)=i^2+i+1$ for $i\geq 1$. Then the process $(X_t,\La_t)$ is recurrent.

  Indeed, it is easy to check that $(\La_t)$ is reversible with respect to a probability measure $(\pi_i)$.
  Taking $\rho(x)=|x|$, we also have $L^{(i)}\rho(x)=\mu_i\rho(x)$ for $|x|\geq 1$ and $i\geq 0$.
  Hence, $\beta_i=i$ for $i\geq 1$ and $\beta_i$ increases to $\infty$ as $i\ra \infty$. Setting $g_i=i+1$ for $i\geq 0$, it is easy to see that there exists $\lambda>0$ such that
  \begin{align*}
   b_0(g_1-g_0)-\mu_0 g_0&\leq -\lambda g_0,\\
   b_i(g_{i+1}-g_i)+a_i(g_{i-1}-g_i)+\mu_i g_i&\leq -\lambda g_i,\qquad i\geq 1.
  \end{align*}
  By Theorem \ref{t-var}, we have $\lambda_0\geq \lambda>0$. Therefore, the process $(X_t,\La_t)$ is recurrent by Theorem \ref{t-4.2} and Remark \ref{rem-var}.
\end{exam}

\section{Summary}
 In this work, we mainly studied the asymptotic stability in probability of regime-switching diffusion processes. Our switching process can be a Markov chain in a finite state space or in an infinite countable state space, and its switching rate can be state-independent and state-dependent. In particular, for switching in an infinite countable state space, we proposed two methods: one is the finite dimensional projection method, another is the principal eigenvalue method. Finite dimensional projection method can deal with state-dependent switching, but its switching rates must be bounded; the principal eigenvalue method can deal with switching with unbounded rates, but can not be used to deal with state-dependent switching at the present stage. Some concrete examples were constructed to show the usefulness and sharpness of our criteria.

\noindent\textbf{Acknowledgement} \ The authors are very grateful to the referees for their valuable suggestions and comments.

\end{document}